\documentclass{amsart}[12pt]
\usepackage[]{amsmath, amsthm, amsfonts, verbatim, amssymb}
\usepackage{tikz}
\usetikzlibrary{matrix,arrows}
\usepackage{url}
\usepackage[all,cmtip]{xy}

\usepackage{xcolor}
\usepackage{hyperref}
\usepackage{graphicx}
\usepackage{pinlabel}

\makeatletter
\@namedef{subjclassname@2020}{%
  \textup{2020} Mathematics Subject Classification}
\makeatother

\newcommand\R{{\mathbb R}}

\def\bZ{{\mathbb Z}}

\def\bC{{\mathbb C}}

\newtheorem {theo}{Theorem}

\newtheorem {lemm}{Lemma}
\newtheorem {exam}{Example}
\theoremstyle{Exmple}
\newtheorem {rem}{Remark}
\newtheorem {defi}{Definition}

\newtheorem {prop}{Proposition}

\def\oz1{d{\overline z}^1}
\def\oz2{d{\overline z}^2}
\def\oz3{d{\overline z}^3}

\def\oI{\overline I}

\def\oz{\overline z}

\def\oGamma{\overline \Gamma}

\def\oIq1{\oI_1\cdots\oI_{q-1}}
\def\oIq2{\oI_1\cdots\oI_{q-2}}

\newcommand{\CPb}{\overline{\mathbb{CP}}{}^{2}}
\newcommand{\CP}{{\mathbb{CP}}{}^{2}}

\title[Complex Ball Quotients and new symplectic $4$-manifolds]
{Complex Ball Quotients and new symplectic $4$-manifolds with nonnegative signatures\\} 

\begin{document}

\author{Anar Akhmedov}
\address{School of Mathematics, 
University of Minnesota, 
Minneapolis, MN, 55455, USA}
\email{akhmedov@math.umn.edu}

\author{S\"{u}meyra Sakall{\i}}
\address{Department of Mathematical Sciences,
University of Arkansas, 
Fayetteville, AR, 72701, USA}
\email{ssakalli@uark.edu}

\author{Sai-Kee Yeung}
\address{Department of Mathematics, 
Purdue University, 
West Lafayette, IN 47907-1395, USA}
\email{yeungs@purdue.edu}

\date{January 15, 2019}


\begin{abstract} We present the various constructions of new symplectic $4$-manifolds with non-negative signatures using the complex surfaces on the BMY line $c_1^2 = 9\chi_h$, the Cartwright-Steger surfaces, the quotients of Hirzebruch's certain line-arrangement surfaces, along with the exotic symplectic $4$-manifolds constructed in \cite{AP2, AS}. In particular, our constructions yield to (i) an irreducible symplectic and infinitely many non-symplectic $4$-manifolds that are homeomorphic but not diffeomorphic to $(2n-1)\CP\#(2n-1)\CPb$ for each integer $n \geq 9$, (ii) the families of simply connected irreducible nonspin symplectic $4$-manifolds that have the smallest Euler characteristics among the all known simply connected $4$-manifolds with positive signatures and with more than one smooth structure. We also construct a complex surface with positive signature from the Hirzebruch's line-arrangement surfaces, which is a ball quotient. \end{abstract}

\maketitle

\section{Introduction}

This article is a continuation of the previous work, carried out in (\cite{A1}, \cite{akhmedov}, \cite{AP1}, \cite{ABBKP}, \cite{AP2}, \cite{AP3}, \cite{AP4}, \cite{AP5}, \cite{AP6}, \cite{AHP}, \cite{AGP}, \cite{AS}, \cite{APS}), on the geography of symplectic $4$-manifolds. For some background and concise history on symplectic geography problem, we refer the reader to the introduction found in \cite{AHP}, \cite{AP4}, and \cite{AGP}. 

Our work here is greatly motivated and influenced by the recent work of Donald Cartwright, Vincent Koziarz, and third author in \cite{CKY} and the earlier work of Gopal Prasad and the third author in \cite{PY, PY1}. The main purpose of our article is to construct new minimal symplectic $4$-manifolds that are interesting with respect to the symplectic geography problem. Starting from Cartwright-Steger surfaces, and their normal covers on Bogomolov-Miyaoka-Yau line $c_1^2 = 9\chi_h$, the Hirzebruch's line-arrangement surfaces and their quotients, by forming their symplectic connected sum with the exotic symplectic $4$-manifolds constructed in \cite{AP2, AS}, or the product $4$-manifolds $\Sigma_{g} \times \Sigma_{h}$, and applying the sequence of Luttinger surgeries along the lagrangian tori, we obtain a family of new symplectic $4$-manifolds with non-negative signatures. As a consequence of our work, we produce (i) an irreducible symplectic and infinitely many non-symplectic $4$-manifolds that are homeomorphic but not diffeomorphic to $(2n-1)\CP\#(2n-1)\CPb$ for each integer $n \geq 9$, (ii) the families of simply connected irreducible nonspin symplectic $4$-manifolds that have the smallest Euler characteristics among the all known simply connected $4$-manifolds with positive signature and with more than one smooth structure. We also construct a complex surface on Bogomolov-Miyaoka-Yau line $c_1^2 = 9\chi_h$ using Hirzebruch's certain line-arrangement surface.

Before stating our main results, let us fix some notations that will be used throughout this paper. Given two $4$-manifolds, $X$\/ and $Y$, we will denote their connected sum by $X\# Y$. For a positive integer $k\geq 2$, the connected sum of $k$\/ copies of $X$\/ will be denoted by $kX$\/. Let $\CP$ denote the complex projective plane, with its standard orientation, and let $\CPb$ denote the underlying smooth $4$-manifold $\CP$ equipped with the opposite orientation. Our main results are the following theorems.

\begin{theo}\label{thm:main1} Let $M$ be $(2n-1)\CP\#(2n-1)\CPb$ for any integer $n \geq 9$. Then there exist an infinite family of irreducible symplectic and an infinite family of irreducible non-symplectic $4$-manifolds that all are homeomorphic but not diffeomorphic to $M$. \end{theo}

The theorem above improves one of the main results of \cite{AP3, AS} where exotic irreducible smooth structures on $(2n-1)\CP\#(2n-1)\CPb$ for $n\geq 25$ and for $n\geq 12$ were constructed, respectively. The next theorem improves the main results of \cite{AP3, AHP, AS} for the positive signature cases.  

\begin{theo}\label{thm:main2}
Let $M$ be one of the following\/ $4$-manifolds.  
\begin{itemize}

\item[(i)] $(2n-1)\CP\#(2n-2)\CPb$ for any integer $n \geq 9$.

\item[(ii)] $(2n-1)\CP\#(2n-3)\CPb$ for any integer $n \geq 10$.

\end{itemize}
Then there exist an infinite family of irreducible symplectic\/ $4$-manifolds and an infinite family of irreducible non-symplectic\/ $4$-manifolds that are homeomorphic but not diffeomorphic to\/ $M$.

\end{theo}

The second theorem above, which deals with the cases of signature equal $1$ and $2$, can be extended to the signature grater than equal $3$ cases as well. 

Let us recall that exotic irreducible smooth structures on $(2n-1)\CP\#(2n-1)\CPb$ for $n\geq 12$, on $(2n-1)\CP\#(2n-2)\CPb$ for $n\geq 14$, on $(2n-1)\CP\#(2n-3)\CPb$ for $n\geq 13$, and on $(2n-1)\CP\#(2n-4)\CPb$ for $n\geq 15$ were constructed in \cite{AS} (see also earlier work in \cite{AP3} and \cite{AHP}). 







Our paper is organized as follows. In Sections~\ref{sec: model} and \ref{sec:Luttinger}, we discuss some background material and collect some building blocks that are needed in our constructions of symplectic $4$-manifolds. In Sections~\ref{sec:ball quotients}, \ref{sec:positive}, \ref{sec:positive CS}, we present the proofs of our main results. A preliminary report on this work has been presented by the first author at Purdue University and by the second author at MPIM and in various research seminars and workshops since November 2018.

\section{Complex surfaces on Bogomolov-Miyaoka-Yau line}\label{sec: model}

\subsection{Fake projective planes}\label{sec: model1}

A fake projective plane is a smooth complex surface which is not the complex projective plane, but has the same Betti numbers as the complex projective plane. The first fake projective plane was constructed by David Mumford in 1979 using p-adic uniformization \cite{Mu}. He also showed that there could only be a finite number of such surfaces. Two more examples were found by Ishida and Kato \cite{IsKa} in 1998, and another by Keum \cite{Ke} in 2006. In 2007 \cite{PY} (see also Addendum \cite{PY1}), the third author and Gopal Prasad almost completely classified fake projective planes by proving that they fall into “28 classes”. Using the arithmeticity of the fundamental group of fake projective planes, and the formula for the covolume of principal arithmetic subgroups, they found twenty eight distinct classes of fake projective planes. For a very small number of classes, they left open the question of existence of fake projective planes in that class, but conjectured that there are none. Finally, Donald Cartwright and Tim Steger verified their conjecture and found all the fake projective planes, up to isomorphism, in each of the 28 classes \cite{CS}.

\begin{exam} In this example, we recall some properties of fake projective plane $M$. We refer the reader to \cite{Y} and also \cite{Ye}, where a complete classification of all smooth surfaces of general type with Euler number $3$ is given. There are $50$ pairs of fake projectives planes as classified in \cite{PY, PY1, CKY} and one Cartwright-Steger surface to be explained in {\bf 2.2}.

For fake projective planes, the Euler characteristic and the Betti numbers of $M$ are $e(M)=3$, $b_{1}(M)=0$ and $b_{2}(M)=1$. $M$ is a minimal complex surface of general type with $\sigma(M) =1$, ${c_{1}}^2(M)=3e(M) = 9$ and $\chi_{h}(M) = 1$. The intersection form of $M$ is odd, and has rank $1$. The fundamental group $\Pi$ of $M$ is a torsion-free cocompact arithmetic subgroup of $PU(2, 1)$, thus $M$ is a ball quotient $B_{\mathbb{C}}^{2}/\Pi$. For $46$ pairs of fake projecitve planes, the canonical line bundle $K_M$ is divisible by $3$, i.e., there is a line bundle $L$ such that $K_M=3L$.  For the remaining four pairs of fake projective planes, we know that $K=3H+\tau$ for some torsion line bundle $\tau$. It was mentioned that the class of $H$ can be represented by a symplectic surface $H$ of self-intersection $1$ (see discussion in \cite{Ko}, pages 212-213), but  notice that Taubes result concerning existence of pseudoholomorphic curves does not apply to the classes $H$ or $2H$ in this case (\cite{L-T}). By considering the classes $pH$ for any positive integer $p \geq 3$, we can produce symplectic surface $H(p)$ of self-intersection $p^2$ and the genus $g(H(p)) = 1 + 1/2(pH \cdot pH + 3H \cdot pH) = 1 + p(p+3)/2$. These symplectic surfaces in $M$ are quite useful. Using the symplectic connected sum operation \cite{gompf}, the pair $(M, H(p))$ (or the covers of $M$ on BMY line) and the knot surged homotopy elliptic surfaces $E(n)_{K}$ along with the symplectic submanifold $S_{K}$ can be used to construct exotic symplectic $4$-manifolds with positive signature and near the Bogomolov-Miyaoka-Yau line $c_1^2 = 9\chi_h$. The symplectic surface $S_{K}$ above in $E(n)_{K}$ is a higher genus section of self-intersection $-n$ resulting  from $-n$ sphere section of $E(n)$ under the the knot surgery along a fibered knot $K$ of genus $g$. To make this construction work, one needs to set $n=p^2$ and $g=1+p(p+3)/2$. Since $\pi_1(E(n)_{K} \setminus S_{K})$ is trivial and $\pi_{1}(H(p))$ surjects into $\pi_{1}(M)$, the resulting symplectic $4$-manifold is simply connected. 
\end{exam}

\subsection{Complex surfaces of Cartwright and Steger}
\label{sec: model2}

The study of enumerating the set of all fake projective planes in the so-called class $\mathcal{C}_{11}$ in the notation of  \cite{PY} led Donald Cartwright and Tim Steger to discover a complex surface with irregularity $q=1$ and Euler characteristic $e=3$, named as Cartwright-Steger surface. Cartwright and Steger showed that a certain maximal arithmetic subgroup $\bar{\Gamma}$ of $PU(2; 1)$ contains a torsion-free subgroup $\Pi$ of index $864$ which has abelianization $\mathbb{Z}^2$. Such subgroup $\Pi$ is unique up to conjugation. Furthermore, for each ineger $n \geq 1$, $\Pi$ has a normal subgroup $\Pi_{n}$ of index $n$. Let $M_{n} = B^{2}(\mathbb{C})/\Pi_{n}$ denote the quotient of a complex hyperbolic space by a torsion free lattice $\Pi_{n}$ of $PU(2; 1)$. The Euler characteristic of $M_{n}$\/ is $e(M_{n})=ne(M_{1})=3n$. $M_{n}$\/ is a minimal complex surface of general type with $\sigma(M_{n}) = n$, ${c_{1}}^2(M_{n})=3e(M_{n}) = 9n$ and $\chi_{h}(M_{n}) = n$. The intersection form of $M_{1}$ is odd, indefinite and modulo torsion is isomorphic to $3\langle 1 \rangle \oplus 2\langle -1 \rangle$. The Betti numbers of $M_{1}$ are: $1, 2, 5, 2, 1$. It is known that the Albanese map of $M_{1}$ gives rise to an Albanese fibration with generic fiber of genus $19$ \cite{CKY}.

\subsection{The covers of Cartwright-Steger surface}
\label{forM1}
Let us recall the following from \cite{CKY}. In one of our constructions we will be using the curves $b(M_c)$ or $ b^{-1}(M_c)$ in Proposition 2.4 of \cite{CKY}. For simplicity, let us consider $D = b(M_c)$.

Recall that in the notation of \cite{CS} and \cite{PY}, the maximal arithmetic lattice considered in this case is denoted by $\oGamma$ summarized in Theorem 1 of \cite{CKY}. The lattice of the Cartwright-Steger surface is denoted by $\Pi$ with generators given by $a_1$, $a_2$, $a_3$ explained in Theorem 2 of \cite{CKY}.

The map $\pi:M=B_{\bC}^2/\Pi\rightarrow B_{\bC}^2/\oGamma$ is a covering map of order $864$.
The quotient $B_{\bC}^2/\oGamma$ is represented by the right hand side of Figure 1 of \cite{CKY}.
$D$ is a component of $\pi^{-1}(D_A)$ in the picture and $\pi^{-1}(D_A)$ is an immersed totally geodesic curve. The singularities of $D$ could only be found in $\pi^{-1}(P_1)$ and $\pi^{-1}(P_2)$. $D$ is a component of genus $4$ in $\pi^{-1}(D_A)$. According to Proposition 2.4 of \cite{CKY},  the only singular points of the curve $D$ is given by a point of normal crossing given by $n_{-1}(D)=2$.

By Proposition 2.4(d) and its proof in \cite{CKY}, $\Pi_M\backslash M$ has genus~4 by the Riemann-Hurwitz formula, and we can find explicit generators $u_i$, $v_i$ of~$\Pi_M$ such that $[u_1,v_1][u_2,v_2][u_3,v_3][u_4,v_4]=1$. When $M=b(M_c)$, the following eight elements generate $\Pi_M$:
\begin{displaymath}
\begin{aligned}
p_1&=a_2^3a_1^{-1}a_3^{-1}j^8a_2^{-2}a_1^{-1}j^4,\\
p_2&=a_3^3a_1a_3^2a_2a_1j^4a_3^{-1}j^8a_3^{-2}a_1^{-1}a_3^{-3},\\
p_3&=j^8a_1^{-1}a_3^{-3}a_2^2j^4a_3^{-2}a_1^{-1}a_3^{-3},\\
p_4&=j^8a_2a_1a_2^{-2}a_1^{-1}j^4a_3^3a_1^2a_2^{-1},\\
\end{aligned}
\quad
\begin{aligned}
p_5&=a_3^3a_1a_3^2j^4a_1^{-1}j^8a_3^2a_1a_2^{-3},\\
p_6&=a_3^3a_1a_2a_1a_3a_2^{-3},\\
p_7&=a_3^3a_1j^8a_1a_2^{-2}a_1^{-1}a_3^2j^4,\\
p_8&=j^4a_3^{-2}j^8a_2a_1a_2a_1a_2^{-2},\\
\end{aligned}
\end{displaymath}
and satisfy the single relation
\begin{displaymath}
p_5^{-1}p_2^{-1}p_5p_1p_3p_8^{-1}p_4p_1^{-1}p_7^{-1}p_6^{-1}p_7p_2p_3^{-1}p_8p_4^{-1}p_6=1.
\end{displaymath}
Here the group elements such as $j$ is given by {\bf 1.1} of \cite{CKY}.
The above gives a set of generators for $\pi_1(\hat{D})$, where  $\hat{D}$ is the normalization  of $D$.

To compute $i_*(\pi_1(D))\subset \pi_1(M)$, where $i:D\rightarrow M$ is the inclusion, note that 
$i_*(\pi_1(D))$ is generated by $i_*p_j, j=1,\dots,8$, together with loops around the nodal point.
The following two elements $\pi_1$ and  $\pi_2$ of $\Pi$ satisfy $\pi(b^{-1}(O))\in b(M_c)$, are taken from the third table on page 41 of the arXiv version of the paper \cite{CKY}:\\
$$\pi_1:=a_3^3a_1a_2^{-1}, \ \ 
\pi_2:=a_2a_1^{-2}a_3^{-1}a_1a_3^{-1}a_1^{-1}a_2^{-2}$$
Consider the subgroup $G_1$ of $G$ given by $G_1 = \, < G \, | \, p_1,p_2,p_3,p_4,p_5,p_6,p_7,p_8,\pi_1\pi_2^{-1}>.$
By using the Magma program, one can verify that $G_1$ is a normal subgroup of $\Pi$ of order $4$.  Moreover, it can be verified that the quotient group is $\bZ_2\times\bZ_2$.  No larger subgroup of $G$ containing $G_1$ could
be found and hence $G_1$ is our candidate $i_*(\pi_1(D))$. The authors are grateful to Donald Cartwright for his help with Magma computations.

In conclusion, we have the following.
\begin{prop}
$D$ is an immersed totally geodesic curve satisfying the following properties.
\begin{enumerate}  
\item The normalization $\hat{D}$ of $D$ is a Riemann surface of genus $4$.\\
\item  $D\cdot D=-1$.\\
\item $i_*\pi_1(D)$ is a normal subgroup of $\pi_1(M)$ of index $4$, 
and $\pi_1(M)/i_*\pi_1(D)=\bZ_2\times\bZ_2$.
\end{enumerate}  
\end{prop}

Denote by $H$ the covering group $\pi_1(M)/i_*\pi_1(D)=\bZ_2\times\bZ_2$ in Proposition 1.
We have
$$1\rightarrow i_*\pi_1(D)\rightarrow \pi_1(M)\rightarrow H.$$
Consider now a normal unramified covering $\tilde{M}$ of $M$ with covering group given by $H$.  Let $p:\tilde{M}\rightarrow M$ be the covering map.  
From construction, $p^{-1}(D)$ consists of four connected components.  Let $E$ be one such connected component.  Then from construction,
inclusion $i_*\pi_1(E)\rightarrow \pi_1(\tilde{M})$ is an isomorphism.   Hence we have

\begin{lemm}\label{lemm1}
$E$ is a curve of self-intersection $-1$ on $\tilde{M}$.  The normalization of $E$ is a Riemann surface of genus $4$.  Moreover,
$i_*\pi_1(E)\rightarrow \pi_1(\tilde{M})$ is an isomorphism.
\end{lemm}

This follows from construction.  Note that a neighborhood of $D$ in $M$ is isomorphic to a neighborhood of $E$ in $\tilde{M}$, as
the covering is a normal covering with $\pi_1(\tilde{M})$ a normal subgroup of $\Pi$.

\begin{lemm}\label{lemm2}
The Chern numbers of $\tilde{M}$ are given by
$c_1^2(\tilde{M})=36$, $c_2(\tilde{M})=12$.
\end{lemm}

This follows from the fact that the Chern numbers involved are multiplicative.

\begin{lemm}\label{lemm3}
$\tilde{M} \#\CPb$ contains a symplectic genus $5$ curve $\Sigma_5$ of self intersection $-2$ that carries the fundamental group of $\tilde{M}\#\CPb$
\end{lemm}

\begin{proof}
It was shown in Lemma \ref{lemm1} that $\tilde{M}$ contains a curve $E$ of self intersection $-1$, whose normalization is a Riemann surface of genus $4$. Since genus is a birational invariant, the genus of $E$ is $4$ as well. We symplectically blow up $E$ at its self intersection, so that it becomes square $-5$ curve and the exceptional sphere $e_1$ intersects it twice. We symplectically smooth two intersections points of the proper transform of $E$ with $e_1$, which gives us genus $5$ symplectic curve $\Sigma_5$ of self intersection $-2$ inside $\tilde{M} \#\CPb$. Since $i_*\pi_1(E)\rightarrow \pi_1(\tilde{M})$ is an isomorphism, we see that  $\Sigma_5$
carries the fundamental group of $\tilde{M}\#\CPb$.
\end{proof}

\section{Luttinger surgery and symplectic cohomology $(2n-3)(\mathbb{S}^2\times \mathbb{S}^2)$}\label{sec:Luttinger}

We briefly review the Luttinger surgery, and collect some symplectic building blocks that will be used later in our constructions. For the details on Luttinger surgery, the reader is referred to the papers \cite{lu} and \cite{ADK}. 

\begin{defi} Let $X$\/ be a symplectic $4$-manifold with a symplectic form $\omega$, and the torus $\Lambda$ be a Lagrangian submanifold of $X$. Given a simple loop $\lambda$ on $\Lambda$, let $\lambda'$ be a simple loop on $\partial(\nu\Lambda)$ that is parallel to $\lambda$ under the Lagrangian framing. For any integer $n$, the $(\Lambda,\lambda,1/n)$ Luttinger surgery\/ on $X$\/ defined to be the $X_{\Lambda,\lambda}(1/n) = ( X - \nu(\Lambda) ) \cup_{\phi} (\mathbb{S}^1 \times \mathbb{S}^1 \times \mathbb{D}^2)$,  the $1/n$\/ surgery on $\Lambda$ with respect to $\lambda$ under the Lagrangian framing. Here $\phi : \mathbb{S}^1 \times \mathbb{S}^1 \times \partial \mathbb{D}^2 \to \partial(X - \nu(\Lambda))$ denotes a gluing map satisfying $\phi([\partial \mathbb{D}^2]) = n[{\lambda'}] + [\mu_{\Lambda}]$ in $H_{1}(\partial(X - \nu(\Lambda))$, where $\mu_{\Lambda}$ is a meridian of $\Lambda$. \end{defi}

It is  shown in \cite{ADK} that $X_{\Lambda,\lambda}(1/n)$ possesses a symplectic form that restricts to the original symplectic form $\omega$ on $X\setminus\nu\Lambda$. The proof of the following lemma is easy to verify and is left to the reader as an exercise.
\begin{lemm}\label{invariants}
\begin{enumerate}  
\noindent \item $\pi_1(X_{\Lambda,\lambda}(1/n)) = \pi_1(X- \Lambda)/N(\mu_{\Lambda} \lambda'^n)$, where 
$N(\mu_{\Lambda} \lambda'^n)$ denote the smallest normal subgroup of $\pi_1(X- \Lambda)$ that contains $\mu_{\Lambda} \lambda'^n$ 
\noindent \item $\sigma(X)=\sigma(X_{\Lambda,\lambda}(1/n))$ and $e(X)=e(X_{\Lambda,\lambda}(1/n))$.
\end{enumerate}
\end{lemm}



\subsection{Luttinger surgeries on product manifolds $\Sigma_{n}\times \Sigma_{2}$ and $\Sigma_{n}\times \mathbb{T}^2$}\label{L}

Recall from \cite{FPS, AP1} that for each integer $n\geq2$, there is a family of irreducible pairwise non-diffeomorphic 4-manifolds $\{Y_n(m)\mid m=1,2,3,\dots\}$ that have the same integer cohomology ring as $(2n-3)(\mathbb{S}^2\times \mathbb{S}^2)$. $Y_n(m)$ are obtained by performing $2n+3$ Luttinger surgeries (cf.\ \cite{ADK, lu}) and a single $m$\/ torus surgery on $\Sigma_2\times \Sigma_n$. These $2n+4$ torus surgeries are performed as follows
\begin{eqnarray}\label{first 8 Luttinger surgeries}
&&(a_1' \times c_1', a_1', -1), \ \ (b_1' \times c_1'', b_1', -1), \ \
(a_2' \times c_2', a_2', -1), \ \ (b_2' \times c_2'', b_2', -1),\\ \nonumber
&&(a_2' \times c_1', c_1', +1), \ \ (a_2'' \times d_1', d_1', +1),\ \
(a_1' \times c_2', c_2', +1), \ \ (a_1'' \times d_2', d_2', +m),
\end{eqnarray}
together with the following $2(n-2)$ additional Luttinger surgeries
\begin{gather*}
(b_1'\times c_3', c_3',  -1), \ \ 
(b_2'\times d_3', d_3', -1), \  \dots  ,\ 
(b_1'\times c_n', c_n',  -1), \ \
(b_2'\times d_n', d_n', -1).
\end{gather*}
Here, $a_i,b_i$ ($i=1,2$) and $c_j,d_j$ ($j=1,\dots,n$) denote the standard loops that generate $\pi_1(\Sigma_2)$ and $\pi_1(\Sigma_n)$, respectively. See Figure~\ref{fig:lagrangian-pair} for a typical Lagrangian tori along which the surgeries are performed.  
\begin{figure}[ht]
\begin{center}
\includegraphics[scale=.89]{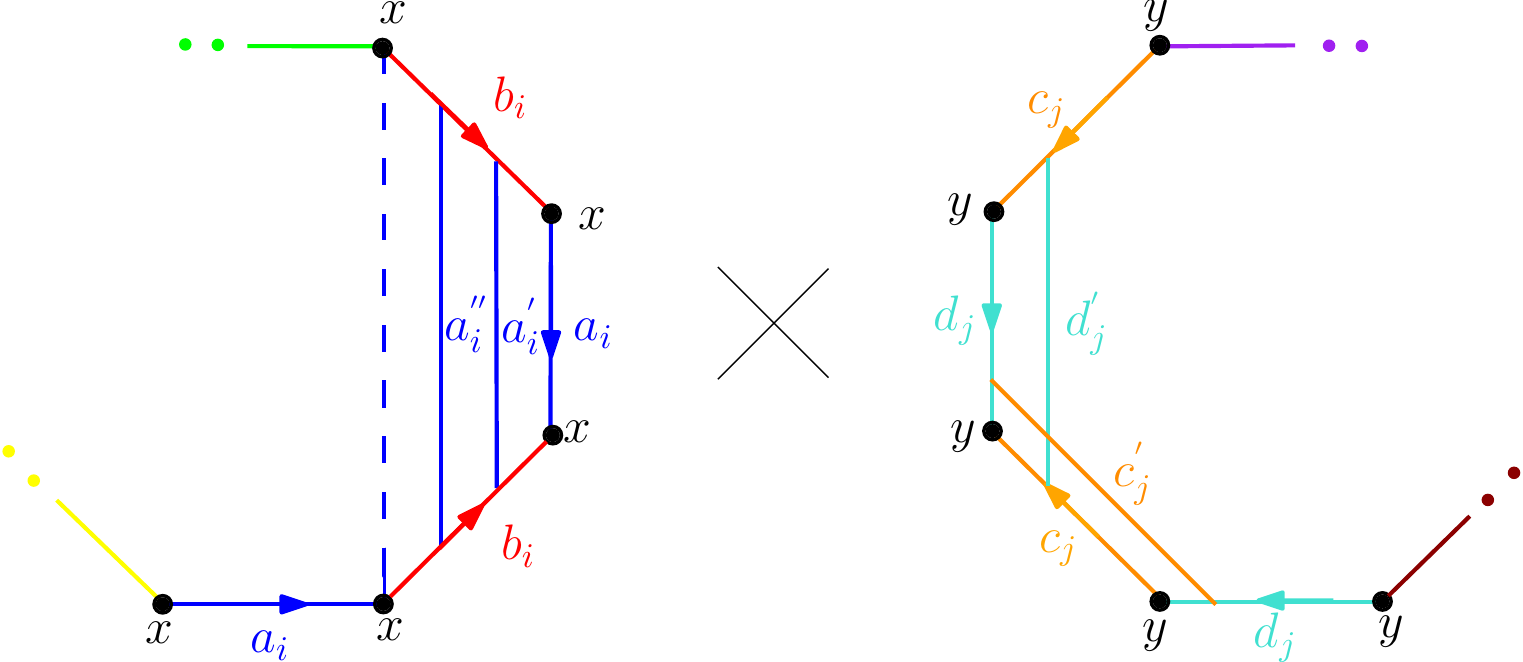}
\caption{Lagrangian tori $a_i'\times c_j'$ and $a_i''\times d_j'$}
\label{fig:lagrangian-pair}
\end{center}
\end{figure}

Since $m$-torus surgery is non-symplectic for $m\geq 2$, the manifold $Y_n(m)$ is symplectic only when $m=1$. Using the Lemma~\ref{invariants}, we see that the Euler characteristic of $Y_n(m)$ is $4n-4$ and its signature is $0$.  $\pi_1(Y_n(m))$ is generated by $a_i,b_i,c_j,d_j$ ($i=1,2$ and $j=1,\dots,n$) and the following relations hold in $\pi_1(Y_n(m))$:  
\begin{gather}\label{Luttinger relations}
[b_1^{-1},d_1^{-1}]=a_1,\ \  [a_1^{-1},d_1]=b_1,\ \  [b_2^{-1},d_2^{-1}]=a_2,\ \  [a_2^{-1},d_2]=b_2,\\ \nonumber
[d_1^{-1},b_2^{-1}]=c_1,\ \ [c_1^{-1},b_2]=d_1,\ \ [d^{-1}_2,b^{-1}_1]=c_2,\ \ [c_2^{-1},b_1]^m=d_2,\\ \nonumber
 [a_1,c_1]=1, \ \ [a_1,c_2]=1,\ \  [a_1,d_2]=1,\ \ [b_1,c_1]=1,\\ \nonumber
[a_2,c_1]=1, \ \ [a_2,c_2]=1,\ \  [a_2,d_1]=1,\ \ [b_2,c_2]=1,\\ \nonumber
[a_1,b_1][a_2,b_2]=1,\ \ \prod_{j=1}^n[c_j,d_j]=1,\\ \nonumber
[a_1^{-1},d_3^{-1}]=c_3, \ \ [a_2^{-1} ,c_3^{-1}] =d_3, \  \dots, \ 
[a_1^{-1},d_n^{-1}]=c_n, \ \ [a_2^{-1} ,c_n^{-1}] =d_n,\\ \nonumber
[b_1,c_3]=1,\ \  [b_2,d_3]=1,\ \dots, \
[b_1,c_n]=1,\ \ [b_2,d_n]=1.
\end{gather}

The surfaces $\Sigma_2\times\{{\rm pt}\}$ and $\{{\rm pt}\}\times \Sigma_n$ in $\Sigma_2\times\Sigma_n$ are not affected by the above Luttinger surgeries, 
and descend to surfaces in $Y_n(m)$. They are symplectic submanifolds in $Y_n(1)$. Let us denote these symplectic submanifolds n $Y_n(1)$ by $\Sigma_2$ and $\Sigma_n$. 
Note that $[\Sigma_2]^2=[\Sigma_n]^2=0$ and $[\Sigma_2]\cdot[\Sigma_n]=1$. Let $\mu(\Sigma_2)$ and $\mu(\Sigma_n)$ denote the meridians of these surfaces in $Y_n(m)$. The above construction easily generalizes to $\Sigma_3\times\Sigma_n$. We will denote the resulting smooth manifold in this case as $Z_n(m)$.

Next, we consider a slightly different construction. Let us fix integers $n \geq 2$, and $m \geq 1$. Let $Y_{n}(1,m)$ denote smooth $4$-manifold obtained by performing the following $2n$ torus surgeries on $\Sigma_n\times \mathbb{T}^2$:
\begin{eqnarray}\label{eq: Luttinger surgeries for Y_1(m)} 
&&(a_1' \times c', a_1', -1), \ \ (b_1' \times c'', b_1', -1),\\  \nonumber
&&(a_2' \times c', a_2', -1), \ \ (b_2' \times c'', b_2', -1),\\  \nonumber
&& \cdots, \ \ \cdots \\ \nonumber
&&(a_{n-1}' \times c', a_{n-1}', -1), \ \ (b_{n-1}' \times c'', b_{n-1}', -1),\\  \nonumber
&&(a_{n}' \times c', c', +1), \ \ (a_{n}'' \times d', d', +m).
\end{eqnarray}

Let $a_i,b_i$ ($i=1,2, \cdots, n$) and $c,d$\/ denote the standard generators of $\pi_1(\Sigma_{n})$ and $\pi_1(\mathbb{T}^2)$, respectively. Since all the torus surgeries listed above are Luttinger surgeries when $m = 1$ and the Luttinger surgery preserves minimality, $Y_{n}(1/p,1/q)$ is a minimal symplectic $4$-manifold. The fundamental group of $Y_{n}(1/p,m/q)$ is generated by $a_i,b_i$ ($i=1,2,3 \cdots, n$) and $c,d$, and the Lemma~\ref{invariants} implies that the following relations hold in $\pi_1(Y_{n}(1,m))$:

\begin{gather}\label{Luttinger relations for Y_1(m)}
[b_1^{-1},d^{-1}]=a_1,\ \  [a_1^{-1},d]=b_1,\ \
[b_2^{-1},d^{-1}]=a_2,\ \  [a_2^{-1},d]=b_2,\\ \nonumber
\cdots,  \ \  \cdots,  \\ \nonumber
[b_{n-1}^{-1},d^{-1}]=a_{n-1},\ \  [a_{n-1}^{-1},d]=b_{n-1},\ \
[d^{-1},b_{n}^{-1}]=c,\ \ {[c^{-1},b_{n}]}^{-m}=d,\\ \nonumber
[a_1,c]=1,\ \  [b_1,c]=1,\ \ [a_2,c]=1,\ \  [b_2,c]=1,\\ \nonumber
[a_3,c]=1,\ \  [b_3,c]=1,\\ \nonumber
\cdots,  \ \  \cdots,  \\ \nonumber
[a_{n-1},c]=1,\ \  [b_{n-1},c]=1,\\ \nonumber
[a_{n},c]=1,\ \  [a_{n},d]=1,\\ \nonumber
[a_1,b_1][a_2,b_2] \cdots [a_n,b_n]=1,\ \ [c,d]=1.
\end{gather}

Let us denote by $\Sigma'_n \subset Y_{n}(1,m)$ a genus $n$ surface that desend from the surface $\Sigma_{n}\times\{{\rm pt}\}$ in $\Sigma_{n}\times \mathbb{T}^2$. 

\section{Construction of a smooth complex algebraic surface with $K^2 = 144$ and $\chi_h = 16$}\label{sec:ball quotients}

In this section, we construct a smooth complex algebraic surface with invariants $K^2 = 144$ and $\chi_h = 16$. This complex surface of general type is on the BMY line $c_1^2 = 9\chi_h$, and thus is a ball quotient. It is obtained as an abelian covering of the complex projective plane branched over an arrangement of $12$ lines shown as in Figure~\ref{fig:H}, known in the literature as the Hesse configuration. Such complex surfaces with bigger invariants, $K^2$ and $\chi_h$, was initially studied by by Friedrich Hirzebruch (for example, see \cite{Hirze}, page 134). Our construction is motivated and similar in spirit to that of Bauer-Catanese in \cite{Main}, where the complex ball quotients obtained from a complete quadrangle arrangement in $\mathbb{CP}^2$.

In $\mathbb{CP}^2$, let us consider the Hesse arrangement $H$, which is a configuration of $9$ points $p_i$ ($1\leq i \leq 9$) and $12$ lines $l_j$ ($1 \leq j \leq 12$), such that each line passes through $3$ of the points $p_i$ and each point lies at the intersection of $4$ of the lines $l_j$ (see Figure~\ref{fig:H}). We blow up $\mathbb{CP}^2$ at the points $p_1, \cdots, p_9$, and denote the blow up map by $\pi: T:=\widehat{\mathbb{CP}^2} \rightarrow \mathbb{CP}^2$. Let $E_i$ be the exceptional divisor corresponding to the blow up at the point $p_i$ for $i=1,\cdots, 9$. In the sequel, we will slightly abuse our notation and denote the proper transform of a line $l_j$ using the same symbol, or $\tilde l_j$ when distinction needed. 

Let us now take the formal sum of the proper transforms $l_j$ of the $12$ lines of the arrangement and the $9$ exceptional divisors $E_i$'s, and denote it by $D$. The divisor $D$ in $T$ has only simple normal crossings. The homology classes of simple closed loops around the $l_j$'s and the $E_i$'s generate $H_1(T-D, \mathbb{Z})$. Let us denote a loop encircling a line $E_i$ or $l_j$ by using the same letter. Then for each $i=1,\cdots,9$, the class of $E_i$ can be written as a sum of the homology classes of $4$ loops around the $4$ lines intersecting $E_i$. To illustrate this, notice that we have $E_1 = l_1 + l_4 + l_7 + l_{10}$ and similar relations hold for the other $E_i$'s. Moreover, the sum of the homology classes of $12$ loops $l_j$'s are $0$, which shows that $H_1(T-D, \mathbb{Z})$ is a free group of rank $11$.

\begin{figure}[ht]
\begin{center}
\includegraphics[scale=.65]{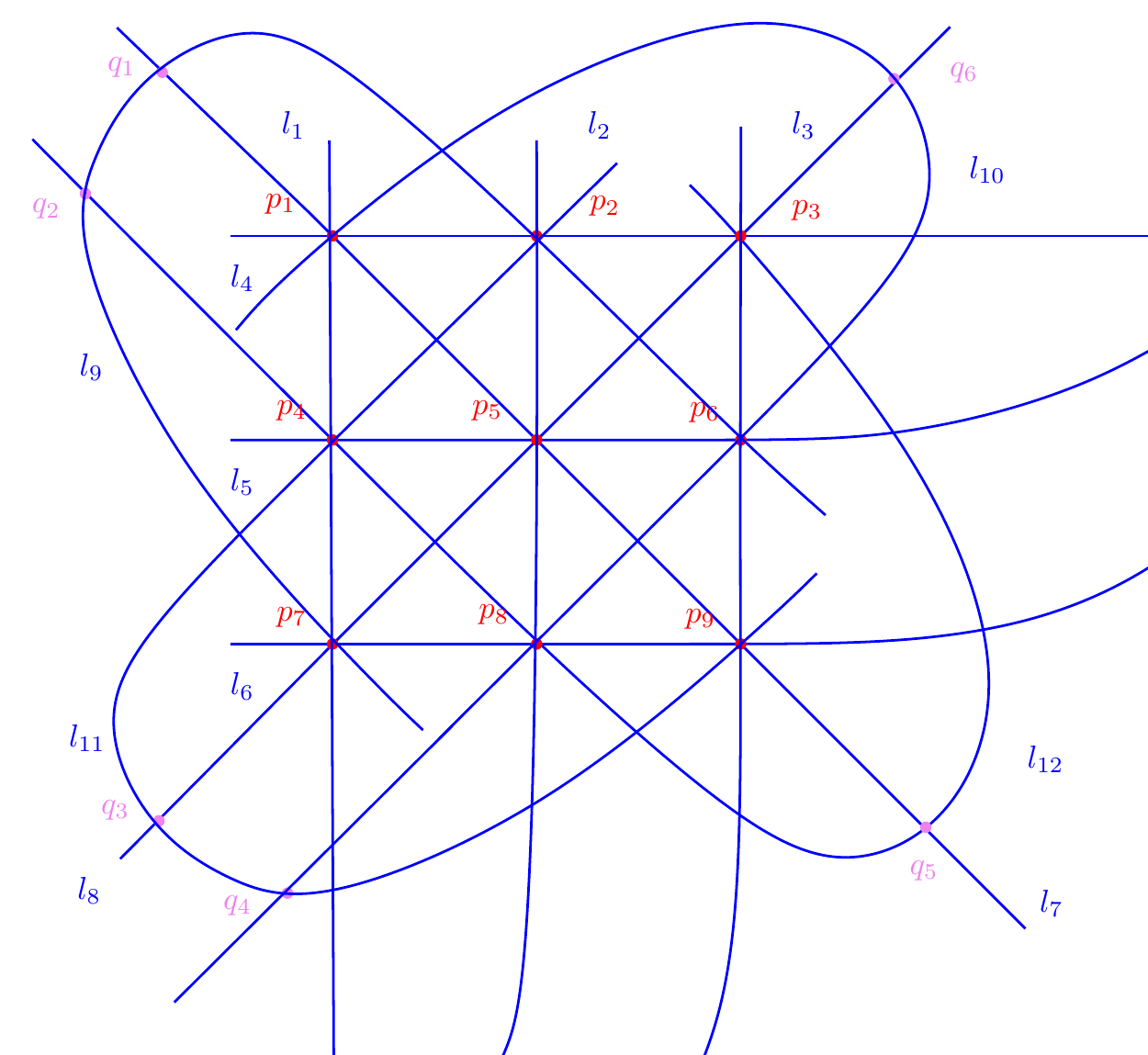}
\caption{Hesse arrangement in $\mathbb{CP}^2$}
\label{fig:H}
\end{center}
\end{figure}

 It is known that a surjective homomorphism $\varphi: \mathbb{Z}^{11} \simeq H_1(T-D,\mathbb{Z}) \rightarrow (\mathbb{Z}/3\mathbb{Z})^2$ determines an abelian $(\mathbb{Z}/3\mathbb{Z})^2$-cover $p:W \rightarrow T=\widehat{\mathbb{CP}^2}$. We need that $p$ is branched exactly in $D$. Notice that there are various epimorphisms satisfying this. To illustrate one, let us define $\varphi$ as follows: 
\begin{align*}
&\varphi(l_1) = \varphi(l_2) = \varphi(l_3) = (1,0),\\
&\varphi(l_4) = (2,0), \varphi(l_5) = (0,1), \varphi(l_6) = (2,2),\\
&\varphi(l_7) = (2,1), \varphi(l_8) = (1,2), \varphi(l_9) = (2,0),\\
&\varphi(l_{10}) = \varphi(l_{11}) = (1,1), \varphi(l_{12}) = (0,2). 
\end{align*}

Then
\begin{align*}
&\varphi(E_1) = \varphi(l_1+l_4+l_7+l_{10}) = (0,2),
&\varphi(E_2) = \varphi(l_2+l_4+l_9+l_{11}) = (0,1),\\
&\varphi(E_3) = \varphi(l_3+l_4+l_{12}+l_8) = (1,1),
&\varphi(E_4) = \varphi(l_1+l_5+l_{11}+l_{12}) = (2,1),\\
&\varphi(E_5) = \varphi(l_2+l_5+l_7+l_8) = (1,1),
&\varphi(E_6) = \varphi(l_3+l_5+l_9+l_{10}) = (1,2),\\
&\varphi(E_7) = \varphi(l_1+l_6+l_8+l_9) = (0,1),
&\varphi(E_8) = \varphi(l_2+l_6+l_{12}+l_{10}) = (1,2),\\
&\varphi(E_9) = \varphi(l_3+l_6+l_7+l_{11}) = (0,1) 
\end{align*}are all nonzero.
Also, $\varphi(l_1+l_2+l_3+l_{7} + l_9+l_{10}) \neq (0,0)$. These conditions ensure that $\varphi$ gives a $(\mathbb{Z}/3\mathbb{Z})^2$ Galois cover branched exactly in $D$ (see Lemma 2.3, part 1 in \cite{Main}, also  \cite{Ku}).

We also note that the following are linearly independent:

\begin{center}
$\varphi(E_1)$ and $\varphi(l_i)$, $i=1,4,7,10$;
$\varphi(E_2)$ and $\varphi(l_i)$, $i=2,4,9,11$; \\
$\varphi(E_3)$ and $\varphi(l_i)$, $i=3,4,12,8$;
$\varphi(E_4)$ and $\varphi(l_i)$, $i=1,5,11,12$;\\ 
$\varphi(E_5)$ and $\varphi(l_i)$, $i=2,5,7,8$;
$\varphi(E_6)$ and $\varphi(l_i)$, $i=3,5,9,10$;\\
$\varphi(E_7)$ and $\varphi(l_i)$, $i=1,6,8,9$;
$\varphi(E_8)$ and $\varphi(l_i)$, $i=2,6,12,10$;\\ 
$\varphi(E_9)$ and $\varphi(l_i)$, $i=3,6,7,11$. 
\end{center}
Moreover, $D$ has simple normal crossings, we deduce that the total space $W$ is smooth (see Lemma 1.4 in \cite{Ku}).

Let us compute some invariants of the surface $W$, and verify that $c_1^2(W)=K_W^2=144$ and $\chi_h(W) = 16$. Let $H$ be the divisor class corresponding to the invertible sheaf $\mathcal{O}(1)$ on $\mathbb{CP}^2$. The canonical sheaf $w_{\mathbb{CP}^2}$ of $\mathbb{CP}^2$ is $\mathcal{O}(-2-1)=\mathcal{O}(-3)$ which corresponds to the canonical divisor $-3H$. Then, the canonical divisor $K_Y$ of $Y$ is $-3H + \sum\limits_{i=1}^{9}E_i$ where we denoted the pullback of $H$ by itself. By using the canonical divisor formula for abelian covers (Proposition 4.2 in \cite{Par}), we compute
\begin{eqnarray*}
K_W &=& \pi^* \Big( (-3H + \sum\limits_{i=1}^{9}E_i) + \frac{2}{3} \sum\limits_{i=1}^{9}E_i +  \frac{2}{3}(12H- 4\sum\limits_{i=1}^{9}(E_i) \Big)\\
&=& \pi^* \Big(5H - \sum\limits_{i=1}^{9}E_i \Big).
\end{eqnarray*}
Since $H \cdot E_i = 0,\, H^2=1$ and $E_i^2 = -1$, the above equality gives $K_W^2 = 9 (25-9) = 144$.

The Euler number $e(W)$ of $W$ can be found as follows.
 \begin{equation*}
e(W) = 9 e(\widehat{\mathbb{CP}^2} = \mathbb{CP}^2 \# 9\overline{\mathbb{CP}^2}) - 6\cdot 21 e(\mathbb{CP}^1)+4 \cdot 48 = 48.
\end{equation*}


Thus $c_1^2(W) = 3 c_2(W)$, and $W$ is a ball quotient. Since $12 \chi_h(W) - c_1^2(W) = e(W)$, we have $\chi_h(W) = 16$. In summary, we proved the  following theorem.

\begin{theo}
There exists a smooth complex algebraic surface $W$ with invariants $c_1^2(W)=144$ and $\chi_h(W) = 16$ constructed as $(\mathbb{Z}/3\mathbb{Z})^2$-cover of $\CP$ branched over the Hesse configuration.
\end{theo}

Now we consider the map $p \circ \pi : W \rightarrow \mathbb{CP}^2$, where $\pi$ is the blow up map, $p$ is the abelian cover. Let us take $p_1$, one of the blown up points in $\mathbb{CP}^2$ which is the intersection point of $l_1,l_4,l_7,l_{10}$ (see Figure~\ref{fig:H}). The pencil of lines in $\mathbb{CP}^2$ passing through $p_1$ lifts to a fibration on $W$. To determine the genus of the generic fiber of this fibration, we take a line $K$ passing through $p_1$ such that its only intersection with the lines $l_1,l_4,l_7,l_{10}$ is $p_1$. In addition, $K$ intersects the remaining $8$ lines of the arrangement. These $8$ intersection points and the point $p_1$ are $9$ branch points on $K$. The preimage of $K-E_{1}$ in $W$, which is the generic fiber of the given fibration, is a degree $3$ cover of $K-E_{1}$ (cf. \cite{BHPV}, p.241), branched at $9$ points. For the determination of the genus $g$ of the surface above $K-E_{1}$, we apply the Riemann-Hurwitz ramification formula
\begin{equation}
2g-2 = 3 (-2) + 9 \cdot 2  \Rightarrow g=7.
\end{equation} 
Therefore, generic fibers are of genus $7$ surfaces. Moreover, there are at least $9$ distinct fibrations in $W$ coming from the points $p_i$'s.

Let us consider the $12$ lines $l_j$ of the Hesse arrangement and determine their inverse images in $W$ under $p \circ \pi$. We observe that on each $l_j$, $j=1, \cdots, 12$, there are $5$ branch points. By the Riemann-Hurwitz formula, we have
\begin{equation}
2g-2 = 3 (-2) + 5 \cdot 2  \Rightarrow g=3.
\end{equation} 
Therefore, they lift to genus $3$ curves. To find their self-intersections, we apply the adjunction formula. Firstly, we note that each $l_j$ is blown up at three points, say $p_k,p_l,p_m$. For its proper transform $\tilde l_j$ in $\widehat{\mathbb{CP}^2}$, we have
\begin{equation}
[\tilde l_j] = H-E_k-E_l-E_m.
\end{equation} 
Thus, 
\begin{eqnarray*}
 K_W \cdot [\Sigma_3] &=& \pi^* \Big((5H - \sum\limits_{i=1}^{9}E_i) \cdot (H-E_k-E_l-E_m) \Big)\\
 &=&3(5-1-1-1) = 6.
\end{eqnarray*}
Using the adjunction formula $2g-2 = 4 = K_W \cdot [\Sigma_3]  + [\Sigma_3]^2$, we have $[\Sigma_3]^2 = -2$. On the other hand, on each exceptional sphere $E_i$, there are $4$ branch points. Thus, their preimages are genus $2$ curves in $W$:
\begin{equation}
2g-2 = 3 (-2) + 4 \cdot 2  \Rightarrow g=2.
\end{equation} 
Similarly as above, 
\begin{equation*}
K_W \cdot [\Sigma_2] = \pi^* \Big((5H - \sum\limits_{i=1}^{9}E_i) \cdot (E_i) \Big)=3
\end{equation*}
and by the adjunction formula we have $2g-2 = 2 = K_W \cdot [\Sigma_2]  + [\Sigma_2]^2$; which shows that $[\Sigma_2]^2 = -1$.

Let us reconsider the pencil of lines in $\mathbb{CP}^2$ passing through $p_1$ and take the line $l_1$. The preimage of its proper transform $\tilde l_1$ is a genus three surface $\Sigma_3$ with self-intersection $-2$ in $W$. The exceptional divisors $E_1$, $E_4$ and $E_7$ intersecting $\tilde l_1$ lift to genus $2$ curves with self-intersections $-1$, each of which intersects $\Sigma_3$ transversally once. Notice that the lift of $E_1$ gives rise to a section, and the union of lifts of the exceptional divisors $E_4$, $E_7$, and the proper transform of intersecting $\tilde l_1$ corresponds to a singular fiber of the given fibration. We symplectically resolve their three transversal intersection points and obtain genus $9$ symplectic submanifold of $W$ with self intersection $+1$. As in Section 2.3 of \cite{AS}, we have the following proposition.

\begin{prop}
\label{genus9}
$W \# \CPb$ contains an embedded symplectic genus $9$ curve $\Sigma_9$ with self intersection $0$.  Furthermore, there is a surjection $f_{*}: \pi_{1}(\Sigma_9) \rightarrow \pi_{1}(W \# \CPb)$. 
\end{prop}

\section{Constructions of Symplectic 4-Manifolds with Positive Signatures from Hirzebruch's line-arrangement surface}\label{sec:positive}

In what follows, we will construct families of simply connected, minimal, symplectic and smooth 4-manifolds with positive signatures, by making use of the complex surface $W \# \CPb$, which we constructed above using the Hesse configuration. By Proposition \ref{genus9}, we know that $W \# \CPb$ contains an embedded symplectic genus $9$ curve of self-intersection zero. We endow $W\#\CPb$ with the symplectic structure induced from the  K\"{a}hler structure. It will be the first building block in our construction, which has the following invariants: $e(W\#\CPb)= 49$, $\sigma(W\#\CPb) = 15$, $c_1^2(W\#\CPb)= 143$ and $\chi_h(W\#\CPb)= 16$. Our second piece will be a minimal, simply connected and symplectic $4$-manifold $X_{g,g+2}$ (\cite{AS}, Theorem 3.12 and Section 5):

\begin{theo}\label{thm2} 
For any integer $g \geq 1$, there exist a minimal symplectic 4-manifold $X_{g,g+2}$ obtained via Luttinger surgery such that
\begin{itemize}
\item[(i)] $X_{g,g+2}$ is simply connected
\item[(ii)] $e(X_{g,g+2})= 4g+2$, $\sigma (X_{g,g+2}) = - 2$, $c_1^{2}(X_{g,g+2}) = 8g-2$, and\/ $\chi(X_{g,g+2}) = g$.
\item[(iii)] $X_{g,g+2}$ contains the symplectic surface $\Sigma$ of genus $2$ with self-intersection $0$ and two genus $g$ surfaces with self-intersection $-1$ intersecting $\Sigma$ positively and transversally.  
\end{itemize}
\end{theo}

\subsection{Symplectic and smooth 4-manifolds with signatures equal to 12}

We now present our first construction of exotic symplectic $4$-manifolds with $\sigma = 12$. Our first building block is the complex surface $W\#\CPb$ containing genus $9$ symplectic surface with self-intersection $0$. The second building block is obtained from the symplectic $4$-manifold $X_{7,9}$, in the notation of Theorem~\ref{thm2}. We will use the fact that $X_{7,9}$ contains a symplectic genus two surface $\Sigma_2$ with self-intersection $0$ and two genus $7$ symplectic surfaces with self intersections $-1$ intersecting $\Sigma_2$ positively and transversally. 

Let us review the construction of $X_{7,9}$ (see \cite{AS} for the details). We take a copy of $\mathbb{T}^2 \times \{pt\}$ and $\{pt\} \times \mathbb{T}^2$ in $\mathbb{T}^2 \times \mathbb{T}^2$ equipped with the product symplectic form, and symplectically resolve the intersection point of these dual symplectic tori. The resolution produces symplectic genus two surface of self intersection $+2$ in $\mathbb{T}^2 \times \mathbb{T}^2$. By symplectically blowing up this surface twice, in $\mathbb{T}^4 \# 2 \overline{\mathbb{CP}}^{2}$, we obtain a symplectic genus 2 surface $\Sigma_2$ with self-intersection $0$, with two $-1$ spheres (i.e. the exceptional spheres resulting from the blow-ups) intersecting it positively and transversally. Next, we form the symplectic connected sum of $\mathbb{T}^4\#2\overline{\mathbb{CP}}^{2}$ with $\Sigma_2 \times \Sigma_7$ along the genus two surfaces  $\Sigma_2$ and $\Sigma_2 \times \{pt\}$. By performing the sequence of appropriate $\pm 1$ Luttinger surgeries on $(\mathbb{T}^4 \# 2 \overline{\mathbb{CP}}^{2}) \#_{\Sigma_2 = \Sigma_2 \times \{pt\}} (\Sigma_2 \times \Sigma_7)$, we obtain the symplectic $4$-manifold $X_{7,9}$ (\cite{AS}). It can be seen from the construction that, $X_{7,9}$ contains a symplectic surface $\Sigma_2$ with self intersection $0$ and two genus $7$ surfaces $S_{1}$ and $S_{2}$ with self intersections $-1$ which have positive and transverse intersections with $\Sigma_2$. Notice that the surfaces $S_{1}$ and $S_{2}$ result from the internal sum of the punctured exceptional spheres in $\mathbb{T}^4 \# 2\overline{\mathbb{CP}}^{2} \setminus \nu (\Sigma_2)$ and the punctured genus 7 surfaces in $\Sigma_2 \times \Sigma_7 \setminus \nu (\Sigma_2 \times \{pt\})$. Moreover, $X_{7,9}$ contains a pair of disjoint Lagrangian tori $T_{1}$ and $T_{2}$ of self-intersections $0$ such that $\pi_{1}(X\setminus(T_{1}\cup T_{2})) = 1$. Note that these Lagrangian tori descend from $\Sigma_2 \times \Sigma_7$, and survive in  $X_{7,9}$ after symplectic connected sum and the Luttinger surgeries. This is because there are at least two pairs of Lagrangian tori in $\Sigma_2 \times \Sigma_7$ that were away from the standard symplectic surfaces $\Sigma_2 \times \{pt\}$ and $ \{pt\} \times \Sigma_7$, and the Lagrangian tori that were used for Luttinger surgeries.  Also, the fact that $\pi_{1}(X_{7,9}\setminus(T_{1}\cup T_{2})) = 1$ is explained in details in \cite{AP3} (see proof of Theorem 8, page 272). Next, we symplectically resolve the intersection of $\Sigma_2$ and one of the genus $7$ surfaces, say $S_1$, in $X_{7,9}$. This produces the genus $9$ surface $\Sigma_9$ of square $+1$ intersecting the other genus $7$  surface $S_{2}$ with self-intersection $-1$. We blow up $\Sigma_9$ at one point (away from its intersection point with $S_{2}$). Thus we obtained a genus $9$ surface $\Sigma_9'$ of square 0 inside $X_{7,9} \# \CPb$.

Since the two symplectic building blocks $W \# \CPb$ and $X_{7,9} \# \CPb$ contain symplectic genus $9$ surfaces of self intersections zero, we can form their symplectic connected sum along these surfaces. Let
\begin{equation*}
Y = (W \# \CPb) \#_{\Sigma_9=\Sigma_9'} (X_{7,9} \# \CPb).
\end{equation*}


\begin{lemm}
The symplectic manifold $Y$ has $e(Y) = 112$, $\sigma(Y) = 12$.
\end{lemm}

\begin{proof}  We have stated the topological invariants of $W \# \CPb$ above and by Theorem~\ref{thm2}, we have $e(X_{7,9}) = 30$, $\sigma(X_{7,9}) = -2$. Applying the symplectic connected sum formula, we compute the topological invariants of $Y$ as above.
\end{proof}

Next, we proceed by following the same lines in \cite{AS}, Section 5 and the references therein. We show that $Y$ is symplectic and simply connected, using Gompf's Symplectic Connected Sum Theorem and Van Kampen's Theorem, respectively. Using Freedman's classification theorem for simply-connected $4$-manifolds, the lemma above and the fact that $W\#\CPb$\/ contains genus two surface of self-intersection $-1$ disjoint from $\Sigma_9$, we conclude that $Y$\/ is homeomorphic to $61\CP\#49\CPb$. Since $Y$ is symplectic, by Taubes's theorem, $Y$ has non-trivial Seiberg-Witten invariant. Next, using the connected sum theorem, we deduce that the Seiberg-Witten invariant of $61\CP\#49\CPb$ is trivial. Therefore, $Y$ is not diffeomorphic to $61\CP\#49\CPb$. Furthermore, $Y$ is a minimal symplectic $4$-manifold by Usher's Minimality Theorem \cite{usher}. Since symplectic minimality implies smooth minimality, $Y$ is also smoothly minimal, and thus is smoothly irreducible.

Moreover, as explained above, $Y$ contains a pair of disjoint Lagrangian tori $T_{1}$ and  $T_{2}$ of self-intersection $0$ such that $\pi_{1}(Y\setminus(T_{1}\cup T_{2})) = 1$. We can perturb the symplectic form on $Y$ in such a way that one of the tori, say $T_{1}$, becomes symplectically embedded. We perform a knot surgery, (using a knot $K$ with non-trivial Alexander polynomial) on $Y$ along $T_{1}$ to obtain irreducible 4-manifold $(Y)_K$ that is homeomorphic but not diffeomorphic to $Y$. By varying our choice of the knot $K$, we can realize infinitely many pairwise non-diffeomorphic, irreducible 4-manifolds, either symplectic or nonsymplectic. (see Theorem 3.7 in \cite{AS})

\subsection{Symplectic and smooth 4-manifolds with signatures equal to 11} 

In this section, we will construct simply connected, minimal, symplectic and smooth $4$-manifolds with signature is equal to $11$ in two different ways. The first construction gives the exotic  $59\mathbb{CP}^{2} \# 48 \overline{\mathbb{CP}}^{2}$ and the second gives the exotic $63\mathbb{CP}^{2} \# 52\overline{\mathbb{CP}}^{2}$. 

In first construction, one of our building block is again $W\#\CPb$, containing a symplectic genus $9$ surface $\Sigma_9$ of square $0$. To obtain the second symplectic building block, we form the symplectic connected sum of $\mathbb{T}^4 \# 2\overline{\mathbb{CP}}^{2}$ with $\Sigma_2 \times \Sigma_6$ along the genus two surfaces  $\Sigma_2$ and $\Sigma_2 \times \{pt\}$. Let

\begin{equation*}
X_{6,8} = (\mathbb{T}^4 \# 2\overline{\mathbb{CP}}^{2}) \#_{\Sigma_2 = \Sigma_2 \times \{pt\}} (\Sigma_2 \times \Sigma_6).
\end{equation*}

Similar to the discussion above, we see that $X_{6,8}$ contains a symplectic surface $\Sigma_2$ with self intersection $0$ and two genus $6$ surfaces with self intersections $-1$ which have positive and transverse intersections with $\Sigma_2$. Furthermore, $X_{6,8}$ contains a symplectic genus $7$ surface $\Sigma_7$ of square $0$ resulting from the internal sum of a punctured genus one surface in $\mathbb{T}^4 \# 2\overline{\mathbb{CP}}^{2} \setminus \nu (\Sigma_2)$ and a punctured genus 6 surface $\Sigma_6$ in $ \Sigma_2 \times \Sigma_6 \setminus \nu (\Sigma_2 \times \{pt\})$. In addition, $\Sigma_7$ intersects  $\Sigma_2$ positively and transversely once. We symplectically resolve this intersection and get symplectic genus 9 surface of self intersection $+2$. We blow it up at two points and hence we obtain symplectic genus 9 surface $\Sigma_9''$ of square 0 inside $X_{6,8} \#2 \CPb$.

Next, we form the symplectic connected sum of $W \#\CPb$ and $X_{6,8} \# 2 \CPb$ along the symplectic genus 9 surfaces $\Sigma_9$ and $\Sigma_9''$ of squares zero. Let

\begin{equation*}   
V= (W \#\CPb) \#_{\Sigma_9 = \Sigma_9''} (X_{6,8} \# 2\CPb).
\end{equation*} 
The invariants of $X_{6,8} \# 2 \CPb$ are as follows. $e(X_{6,8} \#2 \CPb) = 28$, $\sigma(X_{6,8} \#2 \CPb) = -4$, thus we have 

\begin{lemm}
The symplectic manifold $V$ has $e(V) = 109$, $\sigma(V) = 11$.
\end{lemm}

We conclude as above that $V$ is symplectic, simply connected and an exotic copy of $59\mathbb{CP}^{2} \# 48 \overline{\mathbb{CP}}^{2}$ which is also smoothly minimal. As in the previous case, by performing knot surgery, we obtain infinitely many pairwise non-diffeomorphic, irreducible, symplectic and non-symplectic 4-manifolds. 

\vspace{0.1in}

Let us now build another simply connected, minimal, symplectic and smooth 4-manifolds with signature $11$, but with different $\chi$. Our first building block is $Y= (W \# \CPb) \#_{\Sigma_9=\Sigma_9'} (X_{7,9} \# \CPb)$ constructed above. We note that $Y$ contains a genus $2$ surface of self intersection $-1$ coming from $W$, which is not affected by the symplectic connected sum operation in the construction of $Y$. For instance, one can consider $E_6$ in $\widehat{\mathbb{CP}^2}$, which is the exceptional sphere coming from the blow-up of the point $p_6$ in the Hesse arrangement. We have shown that the exceptional spheres lift to genus $2$ surfaces of self intersections $-1$ in $W$. Take one of the preimages of $E_6$  inside $W$. It is a symplectic genus $2$ surface of square $-1$ embedded in $W$, and descends to $Y$ after the symplectic connected sum. Let us denote it by $\Sigma_2'$. On the other hand, we take copies of $\mathbb{T}^2 \times \{pt\}$ and $\{pt\} \times \mathbb{T}^2$ in $\mathbb{T}^2 \times \mathbb{T}^2$ equipped with the product symplectic form, and symplectically resolve the intersection point of these dual symplectic tori. The resolution produces symplectic genus two surface of self intersection $+2$ in $\mathbb{T}^2 \times \mathbb{T}^2$. By symplectically blowing up this surface, we obtain a symplectic genus $2$ surface $\Sigma_2  \subset \mathbb{T}^4 \#  \overline{\mathbb{CP}}^{2}$ of self-intersection $+1$, with the $-1$ exceptional sphere intersecting it positively and transversally. Next, we form the symplectic connected sum of $Y$ with $\mathbb{T}^4\#\overline{\mathbb{CP}}^{2}$ along the genus two surfaces  $\Sigma_2'$ and $\Sigma_2$. Let

\begin{equation*}   
L= Y \#_{\Sigma_2' = \Sigma_2'} (\mathbb{T}^4\#\CPb).
\end{equation*} 
The invariants of $\mathbb{T}^4\#\CPb$ are as follows. $e(\mathbb{T}^4\#\CPb) = 1$, $\sigma(\mathbb{T}^4\#\CPb) = -1$, thus we have 

\begin{lemm}
$e(L) = 117$, $\sigma(L) = 11$.
\end{lemm}

We can conclude that $L$ is symplectic, simply connected, an exotic copy of $63\mathbb{CP}^{2} \# 52\overline{\mathbb{CP}}^{2}$, which is also smoothly minimal. As in the previous case, by performing knot surgery we realize infinitely many pairwise non-diffeomorphic, irreducible, symplectic and nonsymplectic  4-manifolds.

\section{Constructions of exotic $4$-manifolds with nonnegative signatures from Cartwright-Steger surfaces}\label{sec:positive CS}

In this section, we will construct families of simply connected non-spin symplectic and smooth 4-manifolds with nonnegative signatures and small $\chi$. We consider the complex surface $\tilde{M}$ that we constructed from Cartwright-Steger surfaces in Section \ref{forM1}), with $c_1^2(\tilde{M})= 36$ and $e(\tilde{M})=12$. Using the formulas $\sigma = (c_1^2 - 2e)/3$ and $\chi = (e + \sigma)/4$, we have $\sigma(\tilde{M})=\chi(\tilde{M})=4$. Recall that by Lemma \ref{lemm3}, $\tilde{M} \#\CPb$ contains a genus $5$ symplectic curve $\Sigma_5$ of self intersection $-2$ and $i_*\pi_1(\Sigma_5)\rightarrow \pi_1(\tilde{M}\#\CPb)$ being a surjective homomorphism. In our construction of symplectic $4$-manifolds with nonnegative signatures, $\tilde{M} \#\CPb$ along with $\Sigma_5$ will serve as our first building block. For our second building block, we will use the minimal, simply connected and symplectic 4-manifolds $X_{g,g+2}$ and $X_{g,g+1}$  \cite{AP2} (see also \cite{AS}) for which the following theorems hold:

\begin{theo}\label{theorem} 
For any integer $g \geq 1$, there exist a minimal symplectic 4-manifold $X_{g,g+2}$ obtained via Luttinger surgery such that
\begin{itemize}
\item[(i)] $X_{g,g+2}$ is simply connected
\item[(ii)] $e(X_{g,g+2})= 4g+2$, $\sigma (X_{g,g+2}) = - 2$, $c_1^{2}(X_{g,g+2}) = 8g-2$, and\/ $\chi(X_{g,g+2}) = g$.
\item[(iii)] $X_{g,g+2}$ contains the symplectic surface $\Sigma$ of genus $2$ with self-intersection $0$ and $2$ genus $g$ surfaces with self-intersection $-1$ intersecting $\Sigma$ positively and transversally.  
\end{itemize}
\end{theo}

\begin{theo}\label{thm3}  
There exist a minimal symplectic 4-manifold $X_{g,g+1}$ obtained via Luttinger surgery such that
\begin{itemize}
\item[(i)] $X_{g,g+1}$ is simply connected
\item[(ii)] $e(X_{g,g+1})= 4g+1$, $\sigma (X_{g,g+1}) = - 1$, $c_1^{2}(X_{g,g+2}) = 8g-1$, and\/ $\chi(X_{g,g+1}) = g$.
\item[(iii)] $X_{g,g+1}$ contains the symplectic surface $\Sigma$ of genus $2$ with self-intersection $0$, genus $\Sigma_{g+1}$ symplectic surface with self-intersection $0$ intersecting $\Sigma$ positively and transversally.  
\end{itemize}
\end{theo}

\begin{proof}
For the details of the constructions of $X_{g,g+2}$ and  $X_{g,g+1}$, we refer the reader to \cite{AP2} and Section 5 of \cite{AS}.
\end{proof}

\subsection{Symplectic and smooth manifolds with $(\sigma,\chi)=(1,10)$}

To construct simply connected, symplectic and smooth $4$-manifolds with $(\sigma,\chi)=(1,10)$, we use $\tilde{M} \#\CPb$ containing genus $5$ curve $\Sigma_5$ of self intersection $-2$ and $X_{2,4}$ in the notation of Theorem \ref{theorem}. 

For the convenience of the reader, we briefly review the construction of $X_{2,4}$. Take a copy of $\mathbb{T}^2 \times \{pt\}$ and $\{pt\} \times \mathbb{T}^2$ in $\mathbb{T}^2 \times \mathbb{T}^2$ equipped with the product symplectic form, and symplectically resolve the intersection point of these dual symplectic tori. The resolution produces symplectic genus two surface of self intersection $+2$ in $\mathbb{T}^2 \times \mathbb{T}^2$. By symplectically blowing up this surface twice, in $\mathbb{T}^4 \# 2 \overline{\mathbb{CP}}^{2}$, we obtain a symplectic genus 2 surface $\Sigma_2$ with self-intersection $0$, with two $-1$ spheres (i.e. the exceptional spheres resulting from the blow-ups) intersecting it positively and transversally. We also note that $\Sigma_2$ has a dual symplectic torus $\mathbb{T}^2$ of self intersection zero intersecting $\Sigma_2$ positively and transversally at one point. Next, we form the symplectic connected sum of $\mathbb{T}^4\#2\overline{\mathbb{CP}}^{2}$ with $\Sigma_2 \times \Sigma_2$ along the genus two surfaces  $\Sigma_2$ and $\Sigma_2 \times \{pt\}$. By performing the sequence of $8$ appropriate $\pm 1$ Luttinger surgeries on $(\mathbb{T}^4 \# 2 \overline{\mathbb{CP}}^{2}) \#_{\Sigma_2 = \Sigma_2 \times \{pt\}} (\Sigma_2 \times \Sigma_2)$, we obtain the symplectic $4$-manifold $X_{2,4}$.

It can be seen from the construction that there are genus $3$ curves of self intersections $0$ inside $X_{2,4}$. Each of them comes from the internal sum of the one of the punctured tori in $\mathbb{T}^4 \# 2\overline{\mathbb{CP}}^{2} \setminus \nu (\Sigma_2)$ and one of the punctured genus two surfaces in $\Sigma_2 \times \Sigma_2 \setminus \nu (\Sigma_2 \times \{pt\})$. Such a genus 3 curve of square zero intersects $\Sigma_2$ positively and transversally at one point. We symplectically resolve this intersection and obtain a genus $5$ surface $\Sigma_5'$ of square $+2$ in $X_{2,4}$.

Since the two symplectic building blocks $\tilde{M} \#\CPb$ and $X_{2,4}$ contain symplectic genus $5$ surfaces of self intersections $-2$ and $+2$ respectively, we can form their symplectic connected sum along these surfaces $\Sigma_5$ and $\Sigma_5'$. Let
\begin{equation*}
M_{1,10} = (\tilde{M} \#\CPb) \#_{\Sigma_5 = \Sigma_5'} X_{2,4}.
\end{equation*}  

\begin{lemm}
$\sigma(M_{1,10}) = 1$, $\chi_h(M_{1,10}) =10$, $e(M_{1,10}) = 39$ and $c_1^2 (M_{1,10}) = 81$.
\end{lemm}

\begin{proof} 
We have $\sigma(M_{1,10}) = \sigma(\tilde{M} \#\CPb) + \sigma(X_{2,4}) = 3+(-2) = 1$ and $\chi_h(M_{1,10}) = \chi(\tilde{M} \#\CPb) + \chi(X_{2,4}) + (5-1) = 4+2+4= 10$. Using the
formulas $c_1^2 = 3\sigma + 2e$ and $e = 4\chi - \sigma$, we compute $e(M_{1,10})$ and $c_1^2 (M_{1,10})$ as given.
\end{proof}

Let us now show that $M_{1,10}$ is an exotic copy of $19\mathbb{CP}^{2} \# 18\overline{\mathbb{CP}}^{2}$. Notice that $M_{1,10}$ is symplectic and simply connected, which follows from Gompf's Symplectic Connected Sum Theorem \cite{gompf} and Seifert-Van Kampen's Theorem respectively. Using Freedman's classification theorem for simply-connected $4$-manifolds and the lemma above, $M_{1,10}$ is homeomorphic to $19\mathbb{CP}^{2} \# 18\overline{\mathbb{CP}}^{2}$. Since $M_{1,10}$ is symplectic, by Taubes's theorem it has a non-trivial Seiberg-Witten invariant. Next, by appealing to the connected sum theorem for the Seiberg-Witten invariants, we deduce that the Seiberg-Witten invariant of $19\CP\#18\CPb$ is trivial. Thus, $M_{1,10}$ is not diffeomorphic to $19\CP\#18\CPb$. Furthermore, $M_{1,10}$ is a minimal symplectic $4$-manifold by Usher's Minimality Theorem \cite{usher}. Since symplectic minimality implies smooth minimality, $M_{1,10}$ is also smoothly minimal, and thus is smoothly irreducible \cite{HK}. As in Section~\ref{sec:positive}, by performing appropriate generalized torus surgeries, we realize infinitely many pairwise non-diffeomorphic, irreducible, symplectic and nonsymplectic $4$-manifolds (see \cite{AP2} for the further details of such construction).

\subsection{Symplectic and smooth manifolds with $(\sigma,\chi)=(0,9)$}

In this construction, our first building block is again $\tilde{M} \#\CPb$ containing genus $5$ symplectic curve $\Sigma_5$ of self intersection $-2$. For our second building block, we use $X_{1,3}$ in the notation of Theorem \ref{theorem}. 

Let us recall the construction of $X_{1,3}$. In constructing $X_{1,3}$, we first obtain a symplectic genus 2 surface $\Sigma_2$ with self-intersection $0$, with two $-1$ spheres intersecting it positively and transversally in $\mathbb{T}^4 \# 2 \overline{\mathbb{CP}}^{2}$. In addition, there are symplectic tori $\mathbb{T}^2$ of self intersections zero each of which intersects $\Sigma_2$ positively and transversally once. Next, we form the symplectic connected sum of $\mathbb{T}^4\#2\overline{\mathbb{CP}}^{2}$ with $\Sigma_2 \times \Sigma_1$ along the genus two surfaces  $\Sigma_2$ and $\Sigma_2 \times \{pt\}$. By performing the sequence of $6$ appropriate $\pm 1$ Luttinger surgeries on $(\mathbb{T}^4 \# 2 \overline{\mathbb{CP}}^{2}) \#_{\Sigma_2 = \Sigma_2 \times \{pt\}} (\Sigma_2 \times \Sigma_1)$, we obtain the symplectic $4$-manifold $X_{1,3}$. Therefore, we see that $X_{1,3}$ contains a symplectic surface $\Sigma_2$ with self intersection $0$ and two tori $T_{1}$ and $T_{2}$ with self intersections $-1$ which have positive and transverse intersections with $\Sigma_2$. Note that $T_{1}$ and $T_{2}$ result from the internal sum of the punctured exceptional spheres in $\mathbb{T}^4 \# 2\overline{\mathbb{CP}}^{2} \setminus \nu (\Sigma_2)$ and the punctured tori in $\Sigma_2 \times \Sigma_1 \setminus \nu (\Sigma_2 \times \{pt\})$. Moreover, there are genus 2 curves of self intersections 0 inside $X_{1,3}$. Each of them comes from the internal sum of the one of the punctured tori in $\mathbb{T}^4 \# 2\overline{\mathbb{CP}}^{2} \setminus \nu (\Sigma_2)$ and one of the punctured tori in $\Sigma_2 \times \Sigma_1 \setminus \nu (\Sigma_2 \times \{pt\})$. Such a genus 2 curve $\Sigma_2'$ of square zero intersects $\Sigma_2$ positively and transversally at one point. We symplectically resolve the intersections of $\Sigma_2$ with $T_{1}$ and $\Sigma_2$ with $\Sigma_2'$. Thus we obtain a genus 5 surface $\Sigma_5$ of square $+3$ in $X_{1,3}$. By blowing up $\Sigma_5$ at one point, we obtain a genus $5$ surface $\Sigma_5'$ of square $+2$ in $X_{1,3}\# \CPb$.

Since the two symplectic building blocks $\tilde{M} \#\CPb$ and $X_{1,3}\#\CPb$ contain symplectic genus $5$ surfaces of self intersections $-2$ and $+2$ respectively, we can form their symplectic connected sum along these surfaces $\Sigma_5$ and $\Sigma_5'$. Let
\begin{equation*}
M_{0,9} = (\tilde{M} \#\CPb) \#_{\Sigma_5 = \Sigma_5'} (X_{1,3}\#\CPb).
\end{equation*}  

\begin{lemm}
$\sigma(M_{0,9}) = 0$, $\chi(M_{0,9}) =9$, $e(M_{0,9}) = 36$ and $c_1^2 (M_{0,9}) = 72$.
\end{lemm}

\begin{proof} 
We have $\sigma(M_{0,9}) = \sigma(\tilde{M} \#\CPb) + \sigma(X_{1,3}\#\CPb) = 3+(-3) = 0$ and $\chi(M_{0,9}) = \chi(\tilde{M} \#\CPb) + \chi(X_{1,3}\#\CPb) + (5-1) = 4+1+4= 9$. Consequently, we compute $e(M_{0,9})$ and $c_1^2 (M_{0,9})$ as given in the statement.
\end{proof}

As above, we show that $M_{0,9}$ is an exotic copy of $17\mathbb{CP}^{2} \# 17\overline{\mathbb{CP}}^{2}$ and $M_{0,9}$ is also smoothly irreducible. As in Section~\ref{sec:positive}, by performing generalized torus surgeries, we realize infinitely many pairwise non-diffeomorphic, irreducible, symplectic and nonsymplectic  $4$-manifolds (see \cite{AP2} for the further details of such construction).

\subsection{Symplectic and smooth manifolds with $(\sigma,\chi)=(2,10)$}\label{2,10}

In this case, the first symplectic building blocks is $\tilde{M} \#\CPb$ along the genus $5$ curve $\Sigma_5$ of self intersection $-2$. Our the second symplectic building block is $X_{2,3}$ in the notation of Theorem \ref{thm3}, which was constructed in \cite{AP2}. 

Let us recall the construction of $X_{2,3}$. We take a copy of $\mathbb{T}^2 \times \{pt\}$ and the braided torus $T_{\beta}$ representing the homology class  $2[\{pt\} \times \mathbb{T}^2]$ in $\mathbb{T}^2 \times \mathbb{T}^2$ (see ~\cite{AP2}, page 4 for the construction of $T_{\beta}$). The tori $\mathbb{T}^2 \times \{pt\}$ and $T_{\beta}$ intersect at two points. We symplectically blow up one of these intersection points, and symplectically resolve the other intersection point to obtain the symplectic genus two surface of self intersection $0$ in  $\mathbb{T}^4 \# \overline{\mathbb{CP}}^{2}$ (see ~\cite{AP2}, pages 3-4). The symplectic genus $2$ surface $\Sigma_2$ has a dual symplectic torus $\mathbb{T}^2$ of self intersections zero intersecting $\Sigma_2$ positively and transversally at one point. We form the symplectic connected sum of $\mathbb{T}^4 \# \overline{\mathbb{CP}}^{2}$ with $\Sigma_2 \times \Sigma_2$ along the genus two surfaces  $\Sigma_2$ and $\Sigma_2 \times \{pt\}$. By performing the sequence of $4$ appropriate $\pm 1$ Luttinger surgeries on $(\mathbb{T}^4 \#  \overline{\mathbb{CP}}^{2}) \#_{\Sigma_2 = \Sigma_2 \times \{pt\}} (\Sigma_2 \times \Sigma_2)$, we obtain the symplectic $4$-manifold $X_{2,3}$ constructed in \cite{AP2}. It can be seen from the construction that, $X_{2,3}$ contains a symplectic surface $\Sigma_3$ with self intersection $0$, resulting from the internal sum of the punctured torus
in $\mathbb{T}^4 \# \overline{\mathbb{CP}}^{2} \setminus \nu (\Sigma_2)$ and one of the punctured genus two surfaces in $\Sigma_2 \times \Sigma_2 \setminus \nu (\Sigma_2 \times \{pt\})$. $\Sigma_3$ intersects $\Sigma_2$ positively and transversally at one point. (The reader may see Section 5.3 and Figure 7 in \cite{AS} showing the construction steps for a similar case.) We now symplectically resolve their intersection which gives genus five surface $\Sigma_5'$ of self intersection $+2$ in $X_{2,3}$.

Let
\begin{equation*}
M_{2,10} = (\tilde{M} \#\CPb) \#_{\Sigma_5 = \Sigma_5'} (X_{2,3}).
\end{equation*}  

\begin{lemm}\label{l2}
$\sigma(M_{2,10}) = 2$, $\chi_h(M_{2,10}) =10$, $e(M_{2,10}) = 38$ and $c_1^2 (M_{2,10}) = 82$.
\end{lemm}

\begin{proof} 
We have $\sigma(M_{2,10}) = \sigma(\tilde{M} \#\CPb) + \sigma(X_{2,3}) = 3+(-1) = 2$ and $\chi_h(M_{2,10}) = \chi(\tilde{M} \#\CPb) + \chi(X_{2,3}) + (5-1) = 4+2+4= 10$. Consequently, we compute $e(M_{2,10})$ and $c_1^2 (M_{2,10})$.
\end{proof}

Similarly, using the Lemma \ref{l2} and the above mentioned theorems, we see that $M_{2,10}$ is an exotic copy of $19\mathbb{CP}^{2} \# 17\overline{\mathbb{CP}}^{2}$, which is smoothly irreducible.

\subsection{Symplectic and smooth manifolds with $(\sigma,\chi)=(1,9)$}

Similar to the previous cases, we use $\tilde{M} \#\CPb$ containing genus 5 curve $\Sigma_5$ of self intersection $-2$, and $X_{1,2}\#\CPb$ in the notation of Theorem \ref{thm3}, constructed in \cite{AP2}.

To construct $X_{1,2}$, we first obtain a symplectic genus two surface of self intersection $0$ in $\mathbb{T}^4 \# \overline{\mathbb{CP}}^{2}$ as follows. Let us take a copy of $\mathbb{T}^2 \times \{pt\}$ and the braided torus $T_{\beta}$ representing the homology class  $2[\{pt\} \times \mathbb{T}^2]$ in $\mathbb{T}^2 \times \mathbb{T}^2$. The tori $\mathbb{T}^2 \times \{pt\}$ and $T_{\beta}$ intersect at two points. We symplectically blow up one of these two intersection points, and symplectically resolve the other intersection point to obtain the symplectic genus two surface $\Sigma_2$ of self intersection $0$ in  $\mathbb{T}^4 \# \overline{\mathbb{CP}}^{2}$. Note that the exceptional sphere $S^2$ intersects $\Sigma_2$ positively and transversally twice. Next, we form the symplectic connected sum of $\mathbb{T}^4 \# \overline{\mathbb{CP}}^{2}$ with $\Sigma_2 \times \Sigma_1$ along the genus two surfaces  $\Sigma_2$ and $\Sigma_2 \times \{pt\}$. By performing the sequence of $6$ appropriate $\pm 1$ Luttinger surgeries on $(\mathbb{T}^4 \#  \overline{\mathbb{CP}}^{2}) \#_{\Sigma_2 = \Sigma_2 \times \{pt\}} (\Sigma_2 \times \Sigma_1)$, we obtain the symplectic $4$-manifold $X_{1,2}$. It was shown in \cite{AP2}, $X_{1,2}$ is an exotic copy of $\mathbb{CP}^{2} \# 2\overline{\mathbb{CP}}^{2}$. Observe that as a result of the internal sum of the twice punctured sphere $S^2$ in $\mathbb{T}^4 \# \overline{\mathbb{CP}}^{2} \setminus \nu (\Sigma_2)$ and the twice punctured tori in $\Sigma_2 \times \Sigma_1 \setminus \nu (\Sigma_2 \times \{pt\})$, we acquire a symplectic genus $2$ surface of self intersection $-1$ in $X_{1,2}$ intersecting $\Sigma_2$ positively and transversally twice. We symplectically resolve the two intersections and get symplectic genus $5$ surface of square $+3$ in $X_{1,2}$. We blow up this surface at one point and obtain symplectic genus $5$ surface $\Sigma_5'$ of self intersection $+2$ in $X_{1,2}\#\CPb$.
  
Let us define
\begin{equation*}
M_{1,9} = (\tilde{M} \#\CPb) \#_{\Sigma_5 = \Sigma_5'} (X_{1,2}\#\CPb).
\end{equation*}  

\begin{lemm}\label{l3}
$\sigma(M_{1,9}) = 1$, $\chi_h(M_{1,9}) =9$, $e(M_{1,9}) = 35$ and $c_1^2 (M_{1,9}) = 73$.
\end{lemm}

\begin{proof} 
We have $\sigma(M_{1,9}) = \sigma(\tilde{M} \#\CPb) + \sigma(X_{1,2}\#\CPb) = 3+(-2) = 1$ and $\chi (M_{1,9}) = \chi(\tilde{M} \#\CPb) + \chi(X_{1,2}\#\CPb) + (5-1) = 4+1+4= 9$. Consequently, we compute $e(M_{1,9})$ and $c_1^2 (M_{1,9})$ as given.
\end{proof}

Similarly, using the Lemma \ref{l3} and the above mentioned theorems, we show that the minimal symplectic $4$-manifold $M_{1,9}$ is an exotic copy of $17\mathbb{CP}^{2} \# 16\overline{\mathbb{CP}}^{2}$.

\begin{rem} In this remark, we discuss how to obtain a minimal symplectic $4$-manifold with the fundamental group $\mathbb{Z}_{2}$ and the invariants $(\sigma,\chi)=(0,8)$. Since $e = 4\chi - \sigma = 32$, such a symplectic $4$-manifold yields to a homology $15\mathbb{CP}^{2} \# 15\overline{\mathbb{CP}}^{2}$ with $\pi_{1} \cong \mathbb{Z}_{2}$. Since the covering group of the complex surface $M$ (see Proposition 1) is $\mathbb{Z}_2\times \mathbb{Z}_2$, it has a degree two unramified covering. Let us consider the normal unramified covering $M_2$ of $M$ with covering group given by index two subgroup $H'$ of $\pi_1(M)$.  Let $p: M_2\rightarrow M$ be the covering map. Notice that in this case the pull-back of $D$ under this $\mathbb{Z}_2$ covering is not isomorphic to the fundamental group of the ambient manifold, but rather a normal subgroup of index $2$. Using the symplectic pair $(M_{2}\#\CPb, \Sigma_5)$ instead of $(\tilde{M}\#\CPb, \Sigma_5)$, and $(X_{2,3}, \Sigma_5')$ in our above constructions (see \ref{2,10}) leads to a minimal symplectic $4$-manifold with $(\sigma,\chi)=(0,8)$ and $\pi_{1} \cong \mathbb{Z}_{2}$. Previously no such examples were known.
\end{rem}

\section*{Acknowledgments} 
The first author was partially supported by a Simons Research Fellowship and Collaboration Grant for Mathematicians from the Simons Foundation. He would like to thank the Departments of Mathematics at Purdue and at Harvard Universities for their hospitality, where part of this work was completed. The second author would like to thank Max Planck Institute for Mathematics in Bonn for its support and hospitality. The third author is partially supported by NSF grant DMS-$1501282$. All authors would like to thank Donald Cartwright for his help related to Magma computations.


\begin{thebibliography}{99}

\bibitem{akhmedov} A. Akhmedov, 
\textit{Small exotic\/ $4$-manifolds}, 
Algebr. Geom. Topol. \textbf{8} (2008), 1781--1794.  

\bibitem{A1}  A. Akhmedov, 
\textit{Construction of symplectic cohomology $\mathbb{S}^{2}\times \mathbb{S}^{2}$}, 
G\"{o}kova Geometry and Topology Proceedings, {\bf 14} (2007), 36--48.

\bibitem{ABBKP}  A. Akhmedov, S. Baldridge, R. \.{I}. Baykur, P. Kirk and B. D. Park, 
\textit{Simply connected minimal symplectic\/ $4$-manifolds with 
signature less than\/ $-1$}, 
J. Eur. Math. Soc. {\bf 1} (2010), 133--161.

\bibitem{AP1}  A. Akhmedov and B. D. Park,
\textit{Exotic smooth structures on small\/ $4$-manifolds}, 
Invent. Math. \textbf{173} (2008), 209--223.

\bibitem{AP2}  A. Akhmedov and B. D. Park,
\textit{Exotic smooth structures on small\/ $4$-manifolds with odd signatures}, 
Invent. Math. \textbf{181} (2010), 577--603. 

\bibitem{AP3}  A. Akhmedov and B. D. Park,
\textit{New Symplectic\/ $4$-manifolds with positive signature},
Journal of Gokova Geometry and Topology, \textbf{2} (2008), 1--13. 

\bibitem{AP4}  A. Akhmedov and B. D. Park,
\textit{Geography of Simply Connected Spin Symplectic 4-Manifolds},  
Math. Res. Letters, \textbf{17} (2010), no. 3, 483--492.

\bibitem{AP6} A. Akhmedov and B.~D. Park, {\em Geography of simply-connected nonspin symplectic $4$-manifolds with positive signature. II}, Canadian Mathematical Bulletin, 2020: DOI: https://doi.org/10.4153/S0008439520000533.

\bibitem{AP5}  A. Akhmedov and B. D. Park,
\textit{Geography of Simply Connected Spin Symplectic 4-Manifolds, II},  
C. R.  Acad. Sci. Par.  Ser. I., \textbf{357} (2019), 296--298.

\bibitem{AHP} A. Akhmedov, M. Hughes, and B.~D. Park: {\em Geography of simply-connected nonspin $4$-manifolds with positive signature}, Pacific J. Math., {\bf 262} (2), 2013, 257--282. 

\bibitem{AGP}  A. Akhmedov, B. D. Park, and G. Urzua,
\textit{Spin symplectic 4-manifolds near Bogomolov-Miyaoka-Yau line},  
Journal of Gokova Geometry and Topology, \textbf{4} (2010), 55--66.

\bibitem{APS}  A. Akhmedov, B. D. Park, and S. Sakall{\i},
\textit{Exotic smooth structures on connected sums of $\mathbb{S}^{2}\times \mathbb{S}^{2}$},  
preprint (2019).

 \bibitem{AS} A. Akhmedov and S. Sakall{\i}, {\em On the geography of nonspin symplectic 4-manifolds with nonnegative signature}, Topology and its Applications, 206 (2016), 24-45.

\bibitem{ADK}  D. Auroux, S. K. Donaldson and L. Katzarkov, 
\textit{Luttinger surgery along Lagrangian tori and non-isotopy
for singular symplectic plane curves}, Math. Ann. \textbf{326} (2003), 185--203.

\bibitem{BHPV} W. P. Barth, K. Hulek, C. A. M. Peters, A. Van de Ven.: {\em Compact complex surfaces}, Springer-Verlag, Berlin Heidelberg, Second Enlarged Edition, 2004. 

\bibitem{Main} I. C. Bauer and F. Catanese, {\em A Volume Maximizing Canonical Surface In 3-Space}, Comment. Math. Helv., {\bf 83}, 2008, 387--406.

\bibitem{CS}  D. Cartwright and T. Steger,  
\textit{Enumeration of the 50 fake projective planes},  
C. R. Acad. Sci. Paris, Ser. \textbf{348} (2010), 11--13.

\bibitem{CKY} D. Cartwright, V. Koziarz, and S-K., \textit{On the Cartwright-Steger surface}, J. Algebraic Geom. 26 (2017), 655-689{; long arXiv version, arXiv:1412.4137.}

\bibitem{FPS}  R. Fintushel, B. D. Park and R. J. Stern,
\textit{Reverse engineering small\/ $4$-manifolds}, 
Algebr. Geom. Topol. \textbf{7} (2007), 2103--2116.

\bibitem{freedman}  M. H. Freedman,  
\textit{The topology of four-dimensional manifolds}, 
J. Differential Geom. \textbf{17} (1982), 357--453. 

\bibitem{gompf}  R. E. Gompf, 
\textit{A new construction of symplectic manifolds},
Ann. of Math. \textbf{142} (1995), 527--595.

\bibitem{hambleton-kreck}  I. Hambleton and M. Kreck, 
\textit{On the classification of topological\/ $4$-manifolds with finite
fundamental group}, 
Math. Ann. \textbf{280} (1988), 85--104. 

\bibitem{HK}  M. J. D. Hamilton and D. Kotschick, 
\textit{Minimality and irreducibility of symplectic four-manifolds},
Int. Math. Res. Not. \textbf{2006}, Art. ID 35032, 13 pp.

\bibitem{Hirze} F. Hirzebruch, {\em Arrangements of Lines and Algebraic Surfaces}, Arithmetic and geometry : papers dedicated to I.R. Shafarevich on the occasion of his sixtieth birthday /
TN: 982123, volume 2, 1983, 113--140.

\bibitem{IsKa} M.-N. Ishida and F. Kato, 
\textit{The strong rigidity theorem for non-Archimedean uniformization}, 
Tohoku Math. J. \textbf{50} (1998), 537--555.

\bibitem{Ke} J. Keum, 
\textit{A fake projective plane with an order $7$ automorphism}, 
Topology, \textbf{45} (2006), 919--927.

\bibitem{Ko} D. Kotschick, 
\textit{The Seiberg-Witten invariants of symplectic four-manifolds}, 
Seminaire N.Bourbaki, \textbf{812} (1995-96), 195--220.

\bibitem{lu}  K. M. Luttinger,  \textit{Lagrangian tori in\/ $\R^4$},
J. Differential Geom. \textbf{42} (1995), 220--228.

\bibitem{Mu} D. Mumford, \textit{An algebraic surface with $K$ ample, $K^2 = 9$, $p_{g} = q = 0$},
Amer. J. Math. \textbf{101} (1979), 223--244.

\bibitem{Par} R. Pardini, {\em Abelian covers of algebraic varieties}, J. Reine Angew. Math. 417, 1991, 191--214.

\bibitem{PY} G. Prasad and S-K. Yeung, 
\textit{Fake projective planes},
Invent. Math. \textbf{168} (2007), 321--370.

\bibitem{PY1} G. Prasad and S-K. Yeung, 
\textit{Addendum to ``Fake projective planes" Invent. Math. 168, 321-370 (2007)},
Invent. Math. \textbf{182} (2010), 213--227.

\bibitem{taubes}  C. H. Taubes, 
\textit{The Seiberg-Witten invariants and symplectic forms}, 
Math. Res. Lett. \textbf{1} (1994), 809--822.

\bibitem{L-T} T-J Li, C. H. Taubes (Personal Communication).

\bibitem{usher}  M. Usher, 
\textit{Minimality and symplectic sums}, 
Int. Math. Res. Not. \textbf{2006}, Art. ID 49857, 17 pp.

\bibitem{Ye} S-K. Yeung, 
\textit{Classification of surfaces of general type with Euler number $3$},
J. Reine Angew. Math. 679 (2013), 1-22.

\bibitem{Y} S.-K. Yeung, \textit{Foliations associated to harmonic maps on some complex two ball quotients}, Sci. China Math. 60 (2017), 1137-1148; Erradum,  ibid 63(2020), 1645.

\bibitem{Ku} V. Kulikov, {\em Old and new examples of surfaces of general type with $p_g=0$}, Izvestiya: Mathematics, 2004, {\bf 68(5)}, pp.965--1008.


\end{thebibliography}
\end{document}